\documentclass[11pt]{article}

\usepackage[latin1]{inputenc}
\usepackage{t1enc}
\usepackage{amsmath, amssymb, amsfonts, amsthm, amsopn,epsfig}
\usepackage[none]{hyphenat}
\usepackage{float}
\usepackage{graphicx}
\usepackage{caption}
\usepackage[margin=1in]{geometry}
\usepackage[toc,page]{appendix}
\usepackage{epstopdf}
\usepackage{etoolbox}
\usepackage{csquotes}
\usepackage{xspace}
\usepackage{tikz}
\usepackage{tikz-3dplot}
\usepackage{pgf}
\usetikzlibrary{shapes,calc,positioning}
\usepackage{hyperref}
\hypersetup{
	colorlinks=true,
	linkcolor=blue,
	filecolor=magenta,
	urlcolor=cyan,
	citecolor=blue
}
\usepackage{cleveref}
\usepackage{thm-restate}
\usepackage{comment}
\usepackage{bbding}

\theoremstyle{plain}
\newtheorem{theorem}{Theorem}[section]
\newtheorem{definition}[theorem]{Definition}
\newtheorem{corollary}[theorem]{Corollary}
\newtheorem{claim}[theorem]{Claim}
\newtheorem{lemma}[theorem]{Lemma}
\newtheorem{conjecture}[theorem]{Conjecture}

\newtheorem{prop}[theorem]{Proposition}
\Crefname{theorem}{Theorem}{Theorems}
\Crefname{definition}{Definition}{Definitions}
\Crefname{corollary}{Corollary}{Corollaries}
\Crefname{claim}{Claim}{Claims}
\Crefname{lemma}{Lemma}{Lemmas}
\Crefname{conjecture}{Conjecture}{Conjectures}
\Crefname{problem}{Problem}{Problems}
\crefname{prop}{Proposition}{Propositions}

\newcommand{\eps}{\varepsilon}
\newcommand{\zm}{zero-one\xspace} 
\newcommand{\subs}{\subseteq}

\DeclareMathOperator{\polylog}{polylog}
\DeclareMathOperator{\ex}{ex}
\DeclareMathOperator{\forb}{forb}

\title{Large homogeneous submatrices}
\author{
D\'aniel Kor\'andi \thanks{University of Oxford. Email: \texttt{daniel.korandi@maths.ox.ac.uk}. Research supported by SNSF Postdoc.Mobility fellowship P400P2-186686.}
\and
J\'anos Pach \thanks{R\'enyi Institute, Budapest. Email: \texttt{pach@cims.nyu.edu}. Research partially supported by the National Research, Development and Innovation Office, NKFIH, project KKP-133864, and the Austrian Science Fund (FWF), grant Z 342-N31.} \thanks{MIPT Moscow. Research partially supported by  the Ministry of Educational and Science of the Russian Federation in the framework of MegaGrant No.\ 075-15-2019-1926.}
\and
Istv\'an Tomon \thanks{ETH Zurich. Email: \texttt{istvan.tomon@math.ethz.ch}. Research supported by SNSF grant 200021-175573.} \footnotemark[3]
}
\date{}

\begin{document}

\sloppy

\maketitle

\begin{abstract}
A matrix is homogeneous if all of its entries are equal. Let $P$ be a $2\times 2$ \zm matrix that is not homogeneous. We prove that if an $n\times n$ \zm matrix $A$ does not contain $P$ as a submatrix, then $A$ has an $cn\times cn$ homogeneous submatrix for a suitable constant $c>0$. We further provide an almost complete characterization of the matrices $P$ (missing only finitely many cases) such that forbidding $P$ in $A$ guarantees an $n^{1-o(1)}\times n^{1-o(1)}$ homogeneous submatrix. We apply our results to chordal bipartite graphs, totally balanced matrices, halfplane-arrangements and string graphs.
\end{abstract}

\section{Introduction}

Zero-one matrices play an important role in discrete mathematics, as they can be used to represent (bipartite) graphs, hypergraphs, systems of incidences, and many other binary relations. In such settings, the circumstances often force structural restrictions. In this paper, we analyze the structure of matrices that do not contain a given submatrix $P$, and show that forbidding $P$ often forces a large all-0 or all-1 submatrix. 
With a slight abuse of notation, the letter $c$ appearing in different statements stands for unrelated positive constants.

A matrix is \emph{homogeneous} if all of its entries are equal, and \emph{inhomogeneous} otherwise. We will also say that a matrix $A$ {\em contains} another matrix $P$ if $P$ is a submatrix of $A$, and that $A$ is {\em $P$-free} if $P$ is not a submatrix of $A$. Our first result shows that if $P$ is a $2\times 2$ matrix whose entries are not all 0 or all 1, then every $P$-free \zm matrix contains a linear-size homogeneous submatrix.

\begin{theorem} \label{thm:2by2}
    Let $P$ be an inhomogeneous $2\times 2$ \zm matrix. Then every $P$-free $n\times n$ \zm matrix $A$ contains a homogeneous $cn \times cn$ submatrix, for a suitable constant $c>0$. 
\end{theorem}

As we will see below, this result does not hold when $P$ is the all-0 or the all-1 $2\times 2$ matrix.  We can, however, extend \Cref{thm:2by2} to $2\times k$ matrices by making a small sacrifice on the size of the homogeneous submatrix.

\begin{theorem} \label{thm:2byk}
    Let $P$ be a $2\times k$ \zm matrix that does not contain a $2\times 2$ homogeneous submatrix. Then every $P$-free $n\times n$ \zm matrix $A$ contains a homogeneous $n^{1-o(1)}\times cn$ submatrix, for a suitable constant $c>0$.
\end{theorem}

Of course, one can obtain an analogous result for $k\times 2$ matrices by working with the transposes. In particular, we find a homogeneous $n^{1-o(1)}\times n^{1-o(1)}$ submatrix for any $2\times k$ or $k\times 2$ matrix $P$ with no $2\times 2$ homogeneous submatrix. Moreover, every $1\times k$ matrix can be extended to such a $2\times k$ matrix, so this also holds for $1\times k$ and $k\times 1$ matrices.

We should also point out that, as permuting rows or columns does not affect homogeneous submatrices, the same results hold if we only assume that $A$ can be made $P$-free by reordering its rows and columns. 

\medskip
Our problem arises naturally in numerous combinatorial and geometrical settings. When $A$ represents a bipartite graph, $P$ corresponds to a forbidden induced (ordered) subgraph, and a homogeneous submatrix is a complete or empty bipartite subgraph. When $A$ represents an incidence relation in geometry, $P$ is often a geometrically impossible pattern, and a homogeneous submatrix corresponds to two completely intersecting or disjoint families. When $A$ is the incidence matrix of a hypergraph, a homogeneous submatrix gives a set of hyperedges and a completely disjoint or completely contained set of vertices. We list a few specific applications to chordal bipartite graphs, totally balanced matrices, halfplane-arrangements and string graphs in \Cref{sec:applications}.

Several closely related problems have been studied in the literature, including certain Erd\H{o}s-Hajnal type questions and the Tur\'an problem for ordered graphs and forbidden patterns. We discuss some connections and differences in \Cref{sec:remarks}.

\medskip
As mentioned above, it is not true that forbidding any submatrix $P$ forces an almost linear-size homogeneous submatrix.

\begin{definition} \label{simple}
A \zm matrix $P$ is called \emph{acyclic} if every submatrix of $P$ has a row or column containing at most one 1-entry.  The {\em complement} of $P$ is the matrix $P^c$ obtained from $P$ by replacing the 1-entries with 0s and the 0-entries with 1s.
We say that $P$ is \emph{simple} if both $P$ and $P^c$ are acyclic.
\end{definition}

It is easy to see that $P$ is acyclic if and only if the bipartite graph with biadjacency matrix $P$ is acyclic. If $P$ is not simple, then there are $P$-free \zm matrices with only small homogeneous submatrices. The (fairly standard) probabilistic construction will be given in \Cref{sec:tools}.

\begin{prop} \label{prop:nonsimple}
    Let $P$ be a \zm matrix. If $P$ is not simple, then there is a $P$-free $n\times n$ \zm matrix $A$ with no homogeneous $n^{1-\eps}\times n^{1-\eps}$ submatrix for every large enough $n$, where $\eps=\eps(P)$ is a positive constant.
\end{prop}

\Cref{prop:nonsimple} shows that \Cref{thm:2by2,thm:2byk} are optimal in terms of the matrices covered: the remaining $2\times 2$ or $2\times k$ matrices are not simple, so these statements cannot hold for them. In fact, \Cref{thm:2byk,prop:nonsimple} almost completely characterize which forbidden matrices force an almost-linear homogeneous submatrix, because they only miss a finite number of simple matrices. Indeed, a simple $k\times \ell$ matrix can contain at most $k+\ell-1$ 0-entries and at most $k+\ell-1$ 1-entries, so it must satisfy $2k+2\ell -2 \ge k\ell$, or, equivalently, $(k-2)(\ell-2)\le 2$. So, apart from the matrices treated in \Cref{thm:2byk}, only $3\times 3$, $3\times 4$ and $4\times 3$ matrices can be simple.

We believe that a similar statement should hold for the remaining simple matrices, as well. In fact, we make the following stronger conjecture.

\begin{conjecture}\label{conj:simple}
Let $P$ be a simple \zm matrix. Then every $P$-free $n\times n$ \zm matrix contains an $cn\times cn$ homogeneous submatrix, for a suitable constant $c>0$.
\end{conjecture}

Much of the difficulty in our results comes from the ordered structure of matrices. We can obtain better results if we relax the notion of matrices to ``unordered'' matrices, where the order of the rows and the columns does not matter. We can then say that a \zm matrix is {\em unordered $P$-free} if it does not contain any submatrix whose rows and columns can be permuted to obtain $P$. We show that \Cref{thm:2byk} holds for unordered $2\times k$ matrices.

\begin{theorem}\label{thm:graph}
 Let $P$ be a $2\times k$ \zm matrix that does not contain a $2\times 2$ homogeneous submatrix. Then every \emph{unordered $P$-free} $n\times n$ \zm matrix $A$ contains a homogeneous $cn\times cn$ submatrix, for a suitable constant $c>0$.	
\end{theorem}

Results about unordered matrices can be thought of as results about bipartite graphs. In the language of graphs, \Cref{thm:graph} implies the following statement: let $H_{s,t}$ be a star of size $s$ and a star of size $t$ glued together at one of their leaves, and let $H_{s,t}^{*}$ be the union of $H_{s,t}$ and an isolated vertex. Let $H$ be an induced subgraph of $H_{s,t}^{*}$, and let $G=(A\cup B,E)$ be an induced $H$-free bipartite graph with $|A|=|B|=n$. Then there are linear-size subsets $A_{0}\subs A$ and $B_{0}\subs B$ such that $A_{0}\cup B_{0}$ induces either a complete or an empty bipartite graph in $G$. This latter statement has been proved independently by Axenovich, Tompkins, and Weber \cite{ATW}.

\bigskip
Our paper is organized as follows. In \Cref{sec:results}, we state a number of further results, which imply \Cref{thm:2by2,thm:2byk}, but might be of interest on their own. The proof of these (positive) results are given in \Cref{sec:chessboard,sec:corner,sec:general,sec:perm}. Our negative result, \Cref{prop:nonsimple}, is proved in \Cref{sec:tools}. Finally, we prove \Cref{thm:graph} in \Cref{sec:graph}.

We finish the paper with a few applications in \Cref{sec:applications} and some further connections and remarks in \Cref{sec:remarks}.

\section{Forbidden submatrices} \label{sec:results} 

Our first result is about $2\times k$ matrices that contain a 0-entry and a 1-entry in every column, establishing  \Cref{thm:2by2,thm:2byk} in this special case. 

\begin{theorem}\label{thm:1}
Let $P$ be a $2\times k$ \zm matrix without any homogeneous column. Then every $P$-free $n\times n$ \zm matrix contains a $c\frac{n}{k}\times c\frac{n}{k}$ homogeneous submatrix, for a suitable constant $c\ge 10^{-6}$.
\end{theorem}

As rotation and taking complements does not affect the problem, there is essentially one simple $2\times 2$ matrix not covered by this theorem: $Q=\begin{pmatrix}1 & 0  \\ 0 & 0 \end{pmatrix}$. When $P$ cannot be rotated into a $2\times k$ matrix without homogeneous columns, the problem becomes more difficult. $Q$ is the only such matrix where we can prove a linear lower bound.

\begin{theorem}\label{thm:2}
Let $Q=\begin{pmatrix}1& 0  \\ 0 & 0 \end{pmatrix}$. Then every $Q$-free $n\times n$ \zm matrix contains an $cn\times cn$ homogeneous submatrix, for a suitable constant $c\ge 1/20$.
\end{theorem}

In the general case, we can show the following, somewhat weaker result.

\begin{theorem}\label{thm:3}
Let $P$ be a simple $2\times k$ \zm matrix. Then every $P$-free $n\times n$ \zm matrix contains an $n^{1-o(1)}\times cn$ homogeneous submatrix, for a suitable constant $c>0$.
\end{theorem}

\Cref{thm:2by2} then follows from \Cref{thm:1,thm:2}, and \Cref{thm:2byk} is equivalent to \Cref{thm:3}.
Theorems \ref{thm:1}, \ref{thm:2} and \ref{thm:3} are proved in Sections \ref{sec:chessboard}, \ref{sec:corner}, and \ref{sec:general}, respectively.

\medskip

Note that \Cref{conj:simple} is invariant under taking complements: the complement of a $P$-free \zm matrix $A$ is $P^c$-free, and if $P$ is simple, then so is $P^c$. We may therefore assume that 0 is the majority entry in $A$, and then try to find a large all-0 submatrix in it. Indeed, this is the approach we take to prove \Cref{thm:1,thm:2,thm:3}. More generally, we believe that the following strengthening of \Cref{conj:simple} might also be true.

\begin{conjecture} \label{conj:acyclic}
Let $P$ be an acyclic \zm matrix. Then for every $\eps>0$, there is a $\delta>0$, such that every $P$-free $n\times n$ \zm matrix with at least $\eps n^2$ 0-entries contains a $\delta n\times\delta n$ all-0 submatrix.
\end{conjecture}

Another immediate corollary of this conjecture would be the following:

\begin{conjecture} \label{conj:acyclicpair}
Let $P$ be an acyclic \zm matrix. Then every $n\times n$ \zm matrix that is both $P$-free and $P^c$-free contains an $cn\times cn$ homogeneous submatrix, for a suitable constant $c>0$.
\end{conjecture}

We can prove these conjectures in the special case when $P$ has no column with more than one 1-entries.

\begin{theorem}\label{thm:perm}
Let $P$ be a \zm matrix such that every column of $P$ has at most one 1-entry. Then every $n\times n$ \zm matrix that is both $P$-free and $P^c$-free contains a $cn\times cn$ homogeneous submatrix, for some $c>0$.
\end{theorem}

Note that \Cref{thm:1} can also be obtained, with slightly weaker constants, as a corollary of this result (by applying \Cref{thm:perm} to the concatenation of $P$ and $P^c$). The proof can be found in \Cref{sec:perm}.

\section{Notation, preliminaries--Proof of \Cref{prop:nonsimple}} \label{sec:tools}

Throughout this paper, we use the following notation. When $A$ is a matrix, $A(i,j)$ denotes the entry in the $i$'th row and $j$'th column. Sometimes we make no distinction between rows and their indices, and refer to the $i$'th row as ``row $i$'' (and, in a similar manner, for columns). We denote the submatrix in the intersection of rows $X$ and columns $Y$ (the submatrix induced by $X$ and $Y$) by $A[X\times Y]$.

We use two natural correspondences between \zm matrices and graphs. The \emph{biadjacency matrix} of a bipartite graph $G=(A\cup B,E)$ is the \zm matrix whose rows are indexed by $A$, columns are indexed by $B$, and the $(a,b)$ entry is 1 for $a\in A$ and $b\in B$ if and only if $ab\in E$. The \emph{incidence matrix} of a graph $G=(V,E)$ is the \zm matrix whose rows are indexed by $V$, columns are indexed by $E$, and the $(v,e)$ entry is 1 if and only if $e$ is incident to $v$.

For two subsets $X,Y\subs [n]$, we write $X<Y$ to denote that $x<y$ for every $x\in X, y\in Y$. When $Y=\{y\}$, we may simply write $X<y$. We systematically omit floor and ceiling signs whenever they are not essential.
\medskip

We start by proving that only acyclic forbidden matrices can force large all-0 submatrices. This also shows that \Cref{conj:acyclic} can only hold for acyclic $P$.

\begin{prop} \label{prop:cyclicmx}
    Let $P$ be a $k\times \ell$ \zm matrix. If $P$ is not acyclic, then there is a $P$-free $n\times n$ \zm matrix $A$ with at least $n^2/2$ 0-entries, but no homogeneous $n^{1-\eps}\times n^{1-\eps}$ submatrices for every large enough $n$, where $\eps=\eps(P)$ is a positive constant.
\end{prop}
\begin{proof}
We may assume that every row and column of $P$ contains at least two 1-entries, as otherwise we can replace $P$ with a submatrix. In particular, we have $k,\ell\ge 2$, and $P$ contains at least $k+\ell$ 1-entries.

Let $A_0$ be a random $n\times 2n$ matrix, where each entry is independently set to 1 with probability $p=\frac{1}{4}n^{-1+\frac{1}{k+\ell}}$, and set to 0 otherwise. First of all, note that the expected number of 1-entries is $2n^2p = \frac{1}{2}n^{1+\frac{1}{k+\ell}} < n^2/8$ if $n$ is large enough, so the probability that $A_0$ has more than $n^2/2$ 1-entries is at most $1/4$. Also, the expected number of submatrices identical to $P$ in $A_0$ is at most $\binom{n}{k}\binom{2n}{\ell}p^{k+\ell}<(2np)^{k+\ell} < n/4$. So with probability at least $1/4$, there are at most $n$ such submatrices. Finally, the probability that $A_0$ contains a homogeneous $m\times m$ matrix is at most
\[ \binom{n}{m}\binom{2n}{m} \left(p^{m^2} + (1-p)^{m^2}\right) \le (2n)^{2m}(1-p)^{m^2} \le e^{4m\log n - pm^2}<1/4 \]
if $4\log n-pm<-2$, which holds for $m=n^{1-\eps}$ whenever $\eps<\frac{1}{k+\ell}$ and $n$ is large enough. So there is an $n\times 2n$ matrix that contains at most $n$ submatrices identical to $P$ and no homogeneous $m\times m$ submatrix. Then we can delete $n$ columns to obtain the $P$-free matrix $A$ we were looking for.
\end{proof}

\begin{proof}[Proof of \Cref{prop:nonsimple}]
Let us apply \Cref{prop:cyclicmx} to $P$ or $P^c$ (whichever is not acyclic) to get $A$ with no homogeneous $n^{1-\eps}\times n^{1-\eps}$ submatrix. Then $A$ or $A^c$ (whichever is $P$-free) will work.
\end{proof}

\begin{definition} We say that a \zm matrix $P$ is \emph{$(\eps,\delta)$-good} if for all $n$, every $n\times n$ $P$-free matrix with at least $\eps n^2$ 0-entries contains a $\delta n\times \delta n$ all-0 submatrix.
\end{definition}

By convention, every matrix contains the $0\times 0$ all-0 submatrix, so every $P$ is $(\eps,0)$-good for every $\eps$.
We prove our main results by showing that certain matrices $P$ are  $(\eps,\delta)$-good for some $\delta>0$. Let us start with a simple case.

\begin{prop} \label{prop:all1row}
The all-1 $1\times k$ matrix $P=\begin{pmatrix}1&\cdots&1\end{pmatrix}$ is $(0,1/k)$-good.
\end{prop}
\begin{proof}
Without assuming anything about the density, we can find an $\frac{n}{k}\times \frac{n}{k}$ all-0 matrix in any $\frac{n}{k}$ rows of an $n\times n$ $P$-free matrix. Indeed, as every row contains at most $k-1$ 1-entries, any $\frac{n}{k}$ rows induce at least $n-\frac{n(k-1)}{k}= \frac{n}{k}$ columns with only 0-entries.
\end{proof}

Of course, if a matrix is $(\eps,\delta)$-good, then it is also $(\eps',\delta)$-good for any $\eps'\ge \eps$. The next lemma shows that adding an all-0 row or column at a border of a matrix does not change goodness.

\begin{lemma} \label{lem:homcolumn}
Let $P$ be a $k\times\ell$ \zm matrix, and let $P'$ be the $k\times (\ell+1)$ matrix obtained from $P$ by appending a new last column of 0-entries. If $P$ is $(\eps,\delta)$-good for some $\eps\ge 0$, then $P'$ is $(2\eps,\delta\eps)$-good.
\end{lemma}
\begin{proof}
Let $A$ be a $P'$-free $n\times n$ matrix with at least $2\eps n^2$ 0-entries. We will find a dense submatrix with an all-0 last column, and then apply the goodness property of $P$ to get the large homogeneous submatrix.

Define $A'$ as the matrix obtained from $A$ by replacing the first $\eps n$ 0-entries of each row by 1-entries (if a row has fewer than $\eps n$ 0-entries, then it becomes a row with all 1's). Then $A'$ has at least $\eps n^2$ 0-entries, so it must contain a column with at least $\eps n$ 0-entries. If column $j$ is such a column, let $I$ be a set of $\eps n$ rows with a 0-entry in the $j$'th column, and let $J_0$ be the first $j-1$ columns. By the definition of $A'$, every row of $B_0=A[I\times J_0]$ has at least $\eps n$ 0-entries, so in total, $B_0$ contains at least $\eps^2 n^2$ 0-entries. Now let $J\subs J_0$ be the $\eps n$ columns with the most 0-entries in them. Then $B=A[I\times J]$ is an $\eps n \times \eps n$ matrix with at least $\eps^3 n^2$ 0-entries.

Note that $B$ is $P$-free, since we could otherwise add 0's in the $j$'th column to get a copy of $P'$ in $A$. As $P$ is $(\eps,\delta)$-good, $B$ must contain a $\delta\eps n \times \delta\eps n$ all-0 submatrix.
\end{proof}

\section{Matrices with no homogeneous columns--Proof of \Cref{thm:1}}\label{sec:chessboard}

In this section, we prove \Cref{thm:1} by showing that every $2\times k$ matrix with no homogeneous columns satisfies \Cref{conj:acyclic}. We first prove this for a special class of ``checkerboard'' matrices. Let $P_{k}$ denote the $2\times k$ matrix defined by $P_{k}(i,j)=1$ if $i+j$ is even, and $P_{k}(i,j)=0$ otherwise. The main concern of this section is to establish that for every $\eps>0$, there is a $\delta>0$ such that $P_{2k}$ is $(\eps,\delta)$-good. The general case will follow easily by observing that every $2\times k$ matrix with no homogeneous columns is a submatrix of $P_{2k}$.

Note that $P_{2k}$ is the concatenation of $k$ copies of $P_2=\begin{pmatrix}1&0\\0&1\end{pmatrix}$. We first consider $P_2$-free families.

\begin{lemma}\label{lem:density}
Let $\eps>0$, and suppose that $A$ is an $n\times n$  \zm matrix with at least $\eps n^2$ 0-entries. Then at least one of the following statements holds.
\begin{enumerate}
 \item \label{item:density1} $A$ contains an $\frac{\eps n}{8}\times \frac{\eps n}{8}$ all-0 submatrix.
 \item \label{item:density2} At least $\frac{\eps^2 n^2}{64}$ different pairs of rows of $A$ contain $P_2$ as a submatrix. 
\end{enumerate}
\end{lemma}

\begin{proof}
Let $t=\frac{\eps n}{8}$. First, we find a $2t\times 2t$ submatrix of $A$ such that its first row and column contain only 0's, moreover, each of these 0-entries is preceded by $2t$ other 0-entries in their rows and columns in $A$.

Let $A'$ be the matrix obtained from $A$ by replacing the first $2t$ 0-entries of each row and column with 1-entries. As at most $4tn$ 0-entries are lost, $A'$ still has at least $\frac{\eps n^2}{2}$ 0-entries. Now let $A''$ be the matrix obtained from $A'$ by replacing the last $2t-1$ 0-entries of each row and column with 1-entries. By the same argument, $A''$ has at least $2n$ 0-entries.

Take a 0-entry in $A''$, say in the $i_1$'th row and $j_1$'th column. By the definition of $A''$, we must have a set $J>j_1$ of $2t-1$ columns such that the $i_1$'th row of $A'$ contains a 0 in these columns, and similarly, there we must have a set $I>i_1$ of $2t-1$ rows such that the $j_1$'th column of $A'$ contains a 0 in these rows. So, the submatrix $A'[(\{i_1\}\cup I)\times (\{j_1\}\cup J)]$ is all-0 in its first row and column. Also, by the definition of $A'$, each row $i\in I$ contains $2t$ 0-entries in $A$ in some columns $Y_i$ preceding the columns of $J$, and similarly, each column $j\in J$ contains $2t$ 0-entries in some rows $X_j<I$.

If $A[I\times J]$ has $t$ rows without a 1-entry, then it contains a $t\times t$ all-0 submatrix, establishing \ref{item:density1}. Hence, we may assume that at least $t$ rows in $A[I\times J]$ contain a 1-entry.

Let $i\in I, j\in J$ be such that $A(i,j)=1$, and look at the $2t\times 2t$ submatrix $A[X_i\times Y_j]$. Again, if this has $t$ rows without a 1-entry, then $A$ contains a $t\times t$ all-0 submatrix, and we are done. Otherwise, there are 1-entries in $t$ different rows of $A[X_i\times Y_j]$. However, if for some $x\in X_i, y\in Y_i$, the entry $A(x,y)$ is 1, then $A[\{x,i\}\times\{y,j\}]=\begin{pmatrix}1&0\\0&1\end{pmatrix}$. For any choice of $(i,j)$, there is such an $(x,y)$ in $t$ different rows, so we find $P_2$ in at least $t^2$ different row pairs, establishing \ref{item:density2}.
\end{proof}

\begin{lemma}\label{lemma:longmtx}
For every $\eps>0$, $P_{2k}$ is $(\eps,\frac{\eps^4}{10^4k})$-good.
\end{lemma}

\begin{proof}
Suppose $A$ is a $P_{2k}$-free $n\times n$ \zm matrix. Let $s=400k/\eps^3$, and divide $A$ into $\frac{n}{s}\times \frac{n}{s}$ blocks $A_{i,j}=A[I_i\times I_j]$, where $I_k$ is the interval $[\frac{(k-1)n}{s}+1,\frac{kn}{s}]$, for every $i,j,k\in[s]$. We say that $(i,j)\in [s]^{2}$ is \emph{heavy} if $A_{i,j}$ contains at least $\frac{\eps n^2}{2s^2}$ 0-entries. If $N$ denotes the number of heavy pairs, then we can bound the number of 0-entries in $A$ as follows:
\[ \eps n^2 \le N\cdot \frac{n^2}{s^2}+s^2 \cdot \frac{\eps n^2}{2s^2}. \]
Consequently, $N\ge \eps s^2/2$.

This means that for some $i_0\in [s]$, there is a set $J\subs [s]$ of at least $t=\eps s/2$ indices such that $(i_0,j)$ is heavy for every $j\in J$.
Let $R_{j}$ be the set of pairs $\{r,q\}\in [n/s]^{(2)}$ such that rows $r$ and $q$ in $A_{i_{0},j}$ together contain $P_2$. If $(i_{0},j)$ is heavy, then by \Cref{lem:density} (applied with parameters $\eps/2$ and $n/s$), either $|R_{j}|\ge \frac{(\eps n/2s)^{2}}{64}$, or $A_{i_{0},j}$ contains an $\frac{\eps n}{16s}\times \frac{\eps n}{16s}$ all-0 submatrix. In the latter case, we are done, so we may assume the former holds for every $j\in J$. Now
\[ t \cdot \frac{(\eps n/2s)^{2}}{64} = \frac{\eps^3 s}{256}\cdot \frac{(n/s)^2}{2} > k\binom{n/s}{2} \]
implies that some pair $\{r,q\}$ is contained in at least $k$ of the sets $R_j$, say in $R_{j_{1}},\dots,R_{j_{k}}$. Then $P_{2k}$ is a submatrix of the union of the matrices $A_{i_0,j_{1}},\dots,A_{i_0,j_{k}}$ in the rows indexed by $r$ and $q$, which is a contradiction.
\end{proof}

\begin{proof}[Proof of \Cref{thm:1}]
Every $2\times k$ matrix $P$ with no homogeneous columns is contained in $P_{2k}$, so if a matrix is $P$-free, then it is also $P_{2k}$-free.\footnote{Note that this observation combined with \Cref{lemma:longmtx} also implies that every such $P$ is $(\eps,\frac{\eps^4}{10^4k})$-good.} Similarly, every $P^c$-free matrix is $P_{2k}$-free, because $P^c$ also has no homogeneous columns.

If $A$ is $P$-free, then $A^c$ is $P^c$-free, so both $A$ and $A^c$ are $P_{2k}$-free. One of $A$ and $A^c$ will contain at least $n^2/2$ 0-entries, so we can apply \Cref{lemma:longmtx} with $\eps=1/2$ to find an $\frac{n}{20^4k}\times \frac{n}{20^4k}$ homogeneous submatrix in $A$.
\end{proof}

Let $f_{k}(\eps)=\sup\{\delta: P_{2k}\mbox{ is }(\eps,\delta)\mbox{-good}\}$, that is, $f_{k}(\eps)$ is the largest $\delta$ such that for every $n$, every $n\times n$ $P_{2k}$-free matrix with $\eps n^{2}$ 0-entries contains a $\delta n\times \delta n$ all-$0$ matrix. One might wonder what the order of $f_{k}(\eps)$ is.  \Cref{lem:density} shows that $f_{1}(\eps)=\Theta(\eps)$ (the upper bound $f_{1}(\eps)\le \eps$ is trivial), while \Cref{lemma:longmtx} implies $f_{k}(\eps)=\Omega(\eps^4)$ for $k\ge 2$. It might seem reasonable to conjecture that $f_{k}(\eps)=\Theta(\eps)$ also holds for $k\geq 2$. However, this is not true, already for $k=2$:  F\"uredi and Hajnal \cite{FH} proved that for every positive integer $m$, there is an $m\times m$ matrix $B$ such that $B$ does not contain either of $\begin{pmatrix}0&0\\0&0\end{pmatrix}$ and $\begin{pmatrix}*&0&*&0\\ 0&*&0&*\end{pmatrix}$ as a submatrix (where $*$ can be either $0$ or $1$), but $B$ contains $\Omega(m\alpha(m))$ 0-entries, where $\alpha(m)$ is the slowly growing inverse Ackermann function. For $\eps=\Omega(\alpha(m)/m)$ and every $n>m$, we can construct the $n\times n$ matrix $A$ by replacing each $1$-entry of $B$ with an $\frac{n}{m}\times\frac{n}{m}$ all-1 matrix, and each $0$-entry of $B$ with an $\frac{n}{m}\times\frac{n}{m}$ all-0 matrix. Then $A$ is $P_4$-free, it has at least $\eps n^2$ 0-entries, but it does not contain any all-0 submatrix with more than $\frac{n}{m}$ rows and columns. As $\frac{1}{m}=O(\frac{\eps}{\alpha(1/\eps)})$, we have $f_{2}(\eps)=O(\frac{\eps}{\alpha(1/\eps)})$.

It would be interesting to determine the true order of magnitude of $f_{2}(\eps)$. We believe the answer should be closer to the upper bound $O(\frac{\eps}{\alpha(1/\eps)})$.

\section{The $2\times 2$ matrix with one 1 in the corner--Proof of \Cref{thm:2}}\label{sec:corner}

In this section, we establish \Cref{thm:2}. As before, we achieve this by showing a density result: we prove that both $Q= \begin{pmatrix}1&0\\0&0\end{pmatrix}$ and its complement satisfy \Cref{conj:acyclic}.

More generally, let $Q_k$ be the the $2\times (k+1)$ matrix such that $Q_{k}(1,i)=1$ for $i=1,\dots,k$, and all other entries are 0. For example, $Q=Q_1$, and $Q_{3}=\begin{pmatrix}1&1&1&0\\0&0&0&0\end{pmatrix}$. \Cref{prop:all1row} and \Cref{lem:homcolumn} easily imply that $Q_k$ is $(\eps,\eps^2/k)$-good for every $\eps$. In this case, we can actually gain a factor of $\eps$:

\begin{lemma}\label{lemma:1cornereasy}
$Q_k$ is $(\eps,\frac{\eps}{2k})$-good for every $\eps\ge 0$.
\end{lemma}
\begin{proof}
Let $A$ be a $Q_k$-free $n\times n$ matrix with at least $\eps n^2$ 0-entries, and let $A'$ be the matrix obtained from $A$ by replacing the first $\eps n/2$ 0-entries in each row and column with 1's. It is easy to see that fewer than $\eps n^2$ entries were changed, so $A'(i_0,j_0)=0$ for some $i_0,j_0\in [n]$. By the definition of $A'$, we then have sets $I,J\subs [n]$ of size $\eps n/2$ such that $I<i_0$ and $J<j_0$, and for every $i\in I$ and $j\in J$, $A(i,j_0)=A(i_0,j)=0$.

Now $A[I\times J]$ is an $\frac{\eps n}{2}\times \frac{\eps n}{2}$ matrix, and as $A$ is $Q_k$-free, it does not contain a $1\times k$ all-1 submatrix. Then, by \Cref{prop:all1row}, it has an  $\frac{\eps n}{2k}\times \frac{\eps n}{2k}$ all-0 submatrix.
\end{proof}

The difficult part is to show that for every $\epsilon>0$, $Q^{c}$ is also $(\eps,\delta)$-good for some $\delta>0$. We prove this in the next lemma.
\begin{lemma} \label{lemma:1corner}
Let $A$ be an $n\times n$ \zm matrix with at least $\eps n^2$ 1-entries. If $A$ does not contain $Q$, then it has an $\frac{\eps n}{18}\times \frac{\eps n}{18}$ all-1 submatrix.
\end{lemma}
\begin{proof}
For an index $i\in [n]$, let $X_i$ denote the submatrix formed by the first $i$ columns of $A$ and let $Y_i$ denote the submatrix of the last $n-i$ columns. Then for some $i$, both $X_i$ and $Y_i$ contain at least $\eps n^2/3$ 1-entries. Note that this implies, in particular, that both $X_i$ and $Y_i$ have at least $\eps n/3$ columns. Also, $X_i$ has at least $\eps n/6$ rows containing at least $\eps n/6$ 1-entries. Indeed, otherwise $X_i$ would contain fewer than $\frac{\eps n}{6}\cdot n+(n-\frac{\eps n}{6})\cdot \frac{\eps n}{6}<\eps n^2/3$ 1-entries in total, which is not the case. Let $X$ be the submatrix of $X_i$ consisting of $\eps n/6$ such rows, and let $Y$ be an $\frac{\eps n}{6}\times\frac{\eps n}{6}$ submatrix of the same rows in $Y_i$.

Now let us define the graph $G$ on the 0-entries of $Y$ as vertices, where we connect two 0-entries by an edge if they are in the same row or the same column of $Y$. For a vertex $v$ in $G$, we define $r(v)$ and $c(v)$ as the row and column of $v$, respectively. We say that a path $v_1\dots v_k$ in $G$ is row-monotone if $r(v_1)\le\dots\le r(v_k)$. This notion is motivated by the following claim.

\begin{claim} \label{claim:copy}
Let $v\in G$ be a vertex of $G$, and let $U=\{u_1,\dots,u_s\}$ be the set of vertices that can be reached from $v$ via a row-monotone path. Then $X$ contains a $t\times\frac{\eps n}{6}$ all-1 submatrix, where $t=|r(U)|$ is the number of different rows of $U$.
\end{claim}
\begin{proof}
Let $u\in U$ be a vertex that can be reached from $v$ via a row-monotone path $u_0u_1\dots u_k$, where $u_0=v$ and $u_k=u$. We are going to show that if $A(r(v),x)=1$ for some column $x$ of $X$, then $A(r(u),x)=1$, as well. In fact, we will show $A(r(u_i),x)=1$ for every $i$, by induction.

Assume this holds for some $i$ (the case $i=0$ is trivial). If $r(u_i)=r(u_{i+1})$, then there is nothing to prove. Otherwise, $r(u_i)<r(u_{i+1})$ and $c(u_i)=c(u_{i+1})$ by the definition of the path. Let us look at the submatrix $A[\{r(u_i),r(u_{i+1})\}\times\{x,c(u_i)\}]$. The entries in the second column are 0 by the definition of $G$, and the top left entry is 1 by assumption. But this submatrix cannot be $Q$, so the bottom left entry $A(r(u_{i+1}),x)$ must also be 1, as needed.

This shows that in $X$, the rows of $r(U)$ have 1-entries wherever $r(v)$ does. The row $r(v)$, like every row of $X$, contains at least $\eps n/6$ 1-entries, so the rows of $r(U)$ together produce a $t\times \frac{\eps n}{6}$ all-1 submatrix.
\end{proof}

\Cref{claim:copy} shows that it would be enough to find a vertex in $G$ that sends monotone paths to at least $\eps n/18$ different rows. The next claim shows that each connected component of $G$ has a vertex $v$ that reaches the whole component via monotone paths.

\begin{claim} \label{claim:comp}
Let $C$ be a connected component of $G$ and let $v\in C$ be a vertex such that $r(v)$ is smallest. Then for every vertex $u\in C$, there is a row-monotone path from $v$ to $u$.
\end{claim}
\begin{proof}
Let $P=v_0\dots v_k$ be a $v$-$u$ walk in $C$ that minimizes $\sum_{w\in P} r(w)$. We will show that $P$ is a row-monotone path. First, we establish the following simple properties for every such minimal path:
\begin{enumerate}
\item $P$ has no three collinear vertices, i.e., there is no $i$ such that $c(v_{i-1})=c(v_i)=c(v_{i+1})$ or $r(v_{i-1})=r(v_i)=r(v_{i+1})$.
\item There is no ``bottom right corner'' in $P$, i.e., there is no $i$ such that $r(v_{i-1})<r(v_i)$ and $c(v_i)>c(v_{i+1})$, and there is no $i$ with $c(v_{i-1})<c(v_i)$ and $r(v_i)>r(v_{i+1})$.
\end{enumerate}
The first property is clear: we would get a better $v$-$u$ walk by simply deleting $v_i$ from $P$. For the second property, suppose there is an $i$ satisfying $r(v_{i-1})<r(v_i)$ and $c(v_i)>c(v_{i+1})$, and look at the $2\times 2$ submatrix $M=A[\{r(v_{i-1}), r(v_i)\} \times \{c(v_{i+1}),c(v_{i})\}]$. Using $c(v_{i-1})=c(v_i)$ and $r(v_i)=r(v_{i+1})$, we see that $P$ contains all entries of this submatrix, except for the top left entry. All vertices in $P$ are 0-entries, so $A(r(v_{i-1}),c(v_{i+1}))= 0$, as well, for otherwise $M=Q$. Then we could replace $v_i$ with the vertex corresponding to this top left entry, and get a new $P$ with smaller $\sum_{w\in P} r(w)$. The other case of property 2 can be proved analogously.

Now let $j$ be the smallest index such that $r(v_j)\ne r(v)$. By the definition of $v$, we have $r(v_j)>r(v_{j-1})$. We can show by induction that from $v_{j-1}$ on, $P$ alternately moves downwards and to the right. Indeed, property 1 shows that the path changes direction after each edge. Now suppose that at some point it moves downwards, i.e., $r(v_{i-1})<r(v_i)$ (as is the case for $i=j$). Then according to property 2, we cannot move towards the left, so we must have $c(v_i)<c(v_{i+1})$. On the other hand, if the path moves to the right, i.e., $c(v_{i-1})<c(v_i)$, then the second case of property 2 forbids a move upwards in the next step, so we must have $r(v_i)<r(v_{i+1})$.

This means that the row coordinates never decrease along $P$, so it is indeed a row-monotone $v$-$u$ walk. In fact, it is a path because of its minimality.
\end{proof}

Now if a component of $G$ has vertices in at least $\eps n/18$ rows, then \Cref{claim:comp,claim:copy} together imply that $X$ contains an $\frac{\eps n}{18}\times \frac{\eps n}{18}$ all-1 submatrix, as needed. The next claim shows that if there is no such component, then we can find a large all-1 submatrix in $Y$, without even forbidding $Q$.

\begin{claim} 
Suppose no component of $G$ has vertices in $\eps n/18$ different rows. Then $Y$ contains an $\frac{\eps n}{18}\times \frac{\eps n}{18}$ all-1 submatrix.
\end{claim}
\begin{proof}
Let $C_1,\dots,C_k$ be the components of $G$, and let $r(C_i)$ and $c(C_i)$ be the row and column sets of $C_i$. Note that all the 0-entries of $Y$ in rows $r(C_i)$ or columns $c(C_i)$ are inside $A[r(C_i)\times c(C_i)]$.

Swapping rows and columns does not affect our statement, so let us reorder the rows and columns of $Y$ so that $r(C_1)$ are the first $|r(C_1)|$ rows, followed by the rows $r(C_2)$, etc., and similarly for columns. This way we get a block-diagonal matrix with blocks $B_i=r(C_i)\times c(C_i)$, where each block has height less than $\eps n/18$ and all the 0-entries are inside the blocks.

Consider the block $B_i$ that touches the $\frac{\epsilon n}{12}$'th (essentially the middle) column of $Y$. If no such block exists, then the right half of $Y$ is an $\frac{\eps n}{6}\times\frac{\eps n}{12}$ all-1 submatrix, so we are done. We know that $B_i$ has fewer than $\eps n/18$ rows, so this block cannot contain entries from both the $\frac{\eps n}{18}$'th and the $\frac{2\eps n}{18}$'th rows of $Y$. If it is disjoint from the $\frac{\eps n}{18}$'th row, then there is an $\frac{\eps n}{18}\times\frac{\eps n}{12}$ all-1 submatrix in the top right corner of $Y$. Otherwise, we find such a submatrix in the bottom left corner of $Y$.
\end{proof}

This completes the proof of \Cref{lemma:1corner}.
\end{proof}

\begin{proof}[Proof of \Cref{thm:2}]
Let $A$ be an $n\times n$ $Q$-free \zm matrix. If $A$ contains at least $2n^2/20$ 0-entries, then by \Cref{lemma:1cornereasy}, it has an $\frac{n}{20}\times\frac{n}{20}$ all-0 submatrix. Otherwise, $A$ contains at least $18n^2/20$ 1-entries, so we can apply \Cref{lemma:1corner} to find an $\frac{n}{20}\times\frac{n}{20}$ all-1 submatrix in $A$.
\end{proof}

The above proof breaks completely if instead of $Q$ we forbid an arbitrary simple $2\times k$ matrix, although most of it (including a weakening of \Cref{claim:copy}) is salvageable in the special case when we forbid $Q_k$. Unfortunately, \Cref{claim:comp} is false even in this case, and we do not see any meaningful way to circumvent it. The best we can do for $Q_k$-free matrices is to find a homogeneous submatrix of size $\frac{cn}{\log n}\times \frac{cn}{\log n}$ using the methods of \Cref{sec:general}.

\section{General $2\times k$ matrices--Proof of \Cref{thm:3}} \label{sec:general}

In this section, we prove \Cref{thm:3} with the help of partial orders. A \emph{comparability graph} is a graph $G$ whose edges correspond to comparable pairs in some partial order on $V(G)$. The key idea in our proof is to introduce partial orders on the rows of $A$ using the forbidden submatrix. To find the homogeneous submatrices, we need to analyze complete bipartite subgraphs in the comparability graphs and their complement. Our bound on the size of the homogeneous submatrix comes from the following result of Fox and Pach \cite{FP}.

\begin{theorem}[Fox, Pach]\label{thm:poset}
	Let $G$ be the union of $k$ comparability graphs $G_{1},\dots,G_{k}$ on the same $n$ vertices. 
	Then either one of the graphs $G_1,\dots, G_k$ or the complement of $G$ 
	contains a complete bipartite graph with parts of size $n2^{-(1+o(1))(\log\log n)^{k}}$.
\end{theorem}

For simplicity, we write $f_{k}(n)=n2^{-(1+o(1))(\log\log n)^{k}}$. We show that if $P$ is $2\times k$ acyclic, then we can find an all-0 matrix of almost linear size in any $P$-free \zm matrix, where the density of $0$-entries is positive.

\begin{lemma}\label{lemma:ordered}
	Let $P$ be an acyclic $2\times k$ \zm matrix. For every $\eps>0$, there is a $\delta$ such that every $P$-free $n\times n$ \zm matrix with at least $\eps n^2$ 0-entries contains an $f_k(\delta n)\times \delta n$ all-0 submatrix.
\end{lemma}
\begin{proof}
Let $\delta=(\frac{\eps n}{16k})^{k+1}$. We will start with some preprocessing on $A$ to find a large submatrix with $k+1$ ``nice'' all-0 columns, such that every row contains many 0-entries between any two nice columns.

Let us call a $(k+1)$-tuple $(c_1,\dots,c_{k+1})$ \emph{nice} for a row $r$ if $c_1<\dots<c_{k+1}$, $A(r,c_i)=0$ for every $i$, and there are at least $\frac{\eps n}{8k}$ 0-entries in $A\big[\{r\}\times [c_i+1,c_{i+1}]\big]$ for $i=1,\dots,k$.	

If the $r$'th row of $P$ contains at least $\frac{\eps n}{2}$ 0-entries, then there are at least $(\frac{\eps n}{8k})^{k+1}$ nice $(k+1)$-tuples for $r$. Indeed, if the columns of the 0-entries in the $r$'th row are $j_{1}<\dots<j_{\ell}$, then every $(k+1)$-tuple $(j_{x_{1}},j_{x_{2}},\dots,j_{x_{k+1}})$ is a nice $(k+1)$-tuple for $r$, whenever $ \frac{\eps ni}{4k}-\frac{\eps n}{16k}\le x_i \le\frac{\eps ni}{4k}+\frac{\eps n}{16k}.$

The number of rows with at least $\frac{\eps n}{2}$ 0-entries is at least $\frac{\eps n}{2}$. Hence, there are at least $(\frac{\eps n}{8k})^{k+2}$  $(k+1)$-tuples in total (with multiplicities) that are nice for some row. As the number of different $(k+1)$-tuples in $[n]$ is less than $n^{k+1}$, some $(k+1)$-tuple $(c_{1},\dots,c_{k+1})$ is nice for at least $(\frac{\eps }{8k})^{k+2}n$ rows $r$. Let $V$ be a set of $(\frac{\eps}{8k})^{k+2}n$ such rows, and let $I_i$ be the interval $[c_i+1,c_{i+1}]$ for $i=1,\dots,k$. Then each row of every matrix $A_i=A[V\times I_i]$ contains at least $\frac{\eps n}{8k}$ 0-entries, and the last column of every $A_i$ is all-0.
\medskip

For every $i\in [k]$, define the graph $G_i$ on vertex set $V$ as follows. We join $a$ and $b$ in $V$ by an edge if the submatrix of $A_i$ induced by rows $\{a,b\}$ does not contain the $i$'th column of $P$. As $A$ is $P$-free, $\bigcup_{i\in [k]}G_{i}$ must be the complete graph on $V$. 

Let us make some observations about these graphs. First of all, if the $i$'th column of $P$ is all-0, then $G_i$ is empty because $A_i$ has an all-0 column. Note also that $P$ can have at most one all-1 column, otherwise it would not be acyclic. Finally (and crucially), if the $i$'th column of $P$ is not homogeneous, then $G_i$ is a comparability graph. Indeed, suppose that the $i$'th column of $P$ is $\begin{pmatrix}0\\1\end{pmatrix}$. For a row $r\in V$, let $X_{r}$ be the set of columns $s$ such that $A_i(r,s)=0$. Then for $r,r'\in V$, where $r<r'$, we have that $r$ and $r'$ are joined by an edge in $G_{i}$ if and only if $X_{r}\subs X_{r'}$. As the relation $\{(r,r') : r<r' \mbox{ and } X_r\subs X_{r'}\}$ is easily seen to be a poset, $G_i$ is indeed a comparability graph. 
A similar argument works if the $i$'th column of $P$ is $\begin{pmatrix}1\\0\end{pmatrix}$.

Let $K\subs [k]$ be the set of inhomogeneous columns in $P$, and let $G=\bigcup_{i\in K} G_i$. By \Cref{thm:poset}, either some $G_i$ or the complement of $G$ contains a complete bipartite graph with parts of size $m=f_k(|V|)$. First suppose that $G_i$ contains $K_{m,m}$ for some $i\in K$. We may assume by symmetry that the $i$'th column of $P$ is $\begin{pmatrix}0\\1\end{pmatrix}$ . Let $v\in V$ be the first row in $A_i$ that appears in this $K_{m,m}$. Then $v$ is adjacent to a set $W\subs V$ of $m$ rows below it, and $X_v\subs X_w$ for every $w\in W$. Recall that $|X_v|\ge \frac{\eps n}{8k}$ by the construction of $A_i$, so $A_i[W\times X_i]$ is an $m\times \frac{\eps n}{8k}$ all-0 submatrix of $A$, as needed. 

Now suppose that the complement of $G$ contains $K_{m,m}$. As $G_i$ is empty for all-0 columns of $P$ and $\bigcup_{i\in [k]}G_{i}$ is the complete graph on $V$, $P$ must have an all-1 column $q$, and the $K_{m,m}$ must be a subgraph of $G_q$. Let $S,T\subs V$ be the two vertex classes of this $K_{m,m}$. By the definition of $G_q$, each column of $A_q$ contains a 1-entry in at most one of $A_q[S\times I_q]$ and $A_q[T\times I_q]$. As $A_q$ has at least $\frac{\eps n}{8k}$ columns, one of $A_q[S\times I_q]$ or $A_q[T\times I_q]$ contains at least $\frac{\eps n}{16k}$ all-0 columns, so $A_{q}$ has an $m\times \frac{\eps n}{16k}$ all-0 submatrix, finishing the proof.
\end{proof}

\begin{proof}[Proof of \Cref{thm:3}]
 Let $A$ be an $n\times n$ $P$-free matrix. As $P$ is simple, both $P$ and $P^c$ are acyclic. So, if $A$ has at least $n^2/2$ 0-entries, we can apply \Cref{lemma:ordered} to $A$ with $P$ and $\eps=1/2$. Otherwise, we can apply the lemma to $A^c$ with $P^c$ and $\eps=1/2$. Either way, we find an $n^{1-o(1)}\times \Omega(n)$ homogeneous submatrix in $A$.
\end{proof}

Note that any improvement in \Cref{thm:poset} would also improve our theorem. However, this alone will not be sufficient to find a linear-size homogeneous submatrix. Indeed, as was shown recently by Kor\'andi and Tomon \cite{KT}, the size of the bipartite graph in \Cref{thm:poset} cannot be replaced by anything larger than $\Omega(n/(\log n)^{k})$.

\medskip

On the other hand, one can find slightly larger all-0 submatrices in \Cref{lemma:ordered} by reducing the number of partial orders we use. For example, we may assume that $K$ in the proof has size at most $k-1$, as otherwise there are no homogeneous columns in $P$ and we can apply \Cref{thm:1}. This immediately guarantees an $f_{k-1}(\delta n)\times \delta n$ homogeneous submatrix.

It is also enough to use just one matrix $A_i$ (and comparability graph $G_i$) for consecutive columns of $P$ if they are the same. For example, if $\ell$ consecutive columns equal $\begin{pmatrix}0\\1\end{pmatrix}$, then one can take $G_i$ to be the comparability graph where two rows $r<r'$ are joined by an edge if $|X_r\setminus X_{r'}|\le (\ell-1)(r-r')$, and use it to embed all $\ell$ columns in $A_i$. With this argument, one can find a $\Omega(\frac{n}{\log n})\times \Omega(n)$ homogeneous submatrix in any $n\times n$ $Q_k$-free \zm matrix.

\section{Matrices without two ones in a column--Proof of \Cref{thm:perm}}\label{sec:perm}

In this section, we prove \Cref{thm:perm}. The main part of our proof is to prove a weaker variant of \Cref{conj:acyclic} for \zm matrices $P$ with no more than one 1-entry per column. Namely, we show that for some $\eps>0$, every $P$-free $n\times n$ matrix $A$ with at least $(1-\eps) n^2$ 0-entries contains an $\eps n\times \eps n$ all-0 submatrix. This will be enough to obtain \Cref{thm:perm} when $A$ is very dense or very sparse in terms of 0-entries. For the range in between, we will use the following result of Alon, Fischer, and Newman \cite{AFN}.

\begin{lemma}[Alon, Fischer, Newman] \label{lemma:dense_sparse}
Let $P$ be a \zm matrix. For every $\eps>0$ there is a $\delta>0$ such that every $P$-free $n\times n$ \zm matrix $A$ has a $\delta n\times \delta n$ submatrix $B$ that has either at most $\eps (\delta n)^2$ or at least $(1-\eps)(\delta n)^2$ 0-entries.
\end{lemma}

\Cref{lemma:dense_sparse} is stated in \cite[Lemma~1.6]{AFN} in a much stronger form in a ``removal lemma''-type setting, with strong quantitative bounds on $\delta$. However, this weak corollary already serves our purposes. Also, let us remark that in the graph world, this lemma corresponds to the well known result of R\"odl \cite{R} that for any graph $H$, induced $H$-free graphs cannot have a uniform edge distribution.

\begin{lemma}\label{lemma:epsilon}
Let $P$ be a \zm matrix such that no column of $P$ contains more than one 1-entry. Then there is an $\eps=\eps(P)>0$ such that every $P$-free $n\times n$ \zm matrix $A$ with at least $(1-\eps)n^2$ 0-entries has an $\eps n\times \eps n$ all-0 submatrix.
\end{lemma}

\begin{proof}
Suppose $P$ has $k-1$ rows and $\ell$ columns. Let $I$ be the $k\times k$ identity matrix, and let $R$ be the $k\times (k\ell)$ matrix that is the concatenation of $\ell$ copies of $I$, i.e., $R(i,i+jk)=1$ for every $i=1,\dots,k$ and $j=0,\dots,\ell-1$, and all other entries of $R$ are 0. It is easy to see that $R$ contains every $(k-1)\times \ell$ matrix with at most one 1-entry per column as a submatrix.\footnote{In fact, they are already contained in the first $k-1$ rows of $R$. We use $R$ for the sake of a simpler presentation.} In particular, every $P$-free matrix is also $R$-free, so it is enough to prove our theorem for $R$ instead of $P$.

Let $s=2(\ell-1)(2k)^k$, $m=\frac{n}{s}$ and $\eps=\frac{1}{8s^{2}k^{2}}$. We will show that if $A$ contains at least $(1-\eps)n^2$ 0-entries but does not have an $\eps n\times \eps n$ all-0 submatrix, then $A$ contains $R$ as a submatrix.

Let us split the first $m$ rows of $A$ into $m\times m$ submatrices $A_i=A\big[[m]\times [(i-1)m+1,im]\big]$ for $i\in [s]$.  Let $T_i$ be the family of $k$-element sets $S\subs [m]$ such that $A_i$ contains a copy of $I$ in the rows indexed by $S$.

\begin{claim} \label{claim:countI}
 $|T_{i}|\geq \frac{1}{2}\left(\frac{m}{2k}\right)^{k}$ for every $i\in [s]$.
\end{claim}

\begin{proof}
Let us consider the matrices
\[ A_{i,j}=A_{i}\left[\left[\frac{(j-1)m}{k}+1,\frac{jm}{k}\right]\times\left[\frac{(j-1)m}{k}+1,\frac{jm}{k}\right]\right] \]
for every $j\in k$. Then $A_{i,1},\dots,A_{i,k}$ are $\frac{m}{k}\times \frac{m}{k}$ submatrices along the diagonal of $A_{i}$.
	
As $\frac{m}{2k}\ge \sqrt{\eps} n$, we know that $A_i$ does not have any $\frac{m}{2k}\times \frac{m}{2k}$ all-0 submatrix. This easily implies that in each $A_{i,j}$, there are at least $\frac{m}{2k}$ 1-entries such that no two share a row or a column. Let $S_{i,j}$ be the set of coordinates of these $\frac{m}{2k}$ 1-entries.

Let us pick an element $(x_{j},y_{j})\in S_{i,j}$ for every $j=1,\dots, k$ (so one 1-entry from each $A_{i,j}$), and consider the $k\times k$ submatrix $B=A_i[\{x_1,\dots,x_k\}\times \{y_1,\dots,y_k\}]$. There are $(\frac{m}{2k})^k$ such submatrices $B$. Also, $B$ has 1-entries in the diagonal, so $B=I$, unless there is another 1-entry in $B$. However, each such 1-entry of $A_i$ can appear in at most $(\frac{m}{2k})^{k-2}$ matrices $B$, because it fixes the choice of $(x_j,y_j)$ for two $j$'s: if $A_i(x,y)=1$, then the 1-entry at $(x,y)$ can only appear in matrices $B$ for which $x_{a}=x$ for $a=\lceil x/k\rceil$ and $y_b=y$ for $b=\lceil y/k\rceil$. As there are at most $\eps n^2$ 1-entries in $A$, we are left with at least
\[ \left(\frac{m}{2k}\right)^{k} - \eps n^2 \left(\frac{m}{2k}\right)^{k-2} =
   \left(\frac{m}{2k}\right)^{k} - \frac{m^2}{8k^2} \left(\frac{m}{2k}\right)^{k-2} =
   \frac{1}{2}\left(\frac{m}{2k}\right)^{k} \]
choices where $B=I$.
\end{proof}

Suppose that $A$ does not contain $R$ as a submatrix. Then every $k$-element set $S\subs [m]$ can appear in at most $\ell-1$ of the sets  $T_1,\dots, T_s$. Indeed, if $S\in T_{i_1}\cap T_{i_2}\cap\dots\cap T_{i_\ell}$, then $A$ contains $R$ as a submatrix in the rows induced by $S$.

Together with \Cref{claim:countI}, this gives
\[ \frac{s}{2}\left(\frac{m}{2k}\right)^{k} \le  \sum_{i=1}^{s}|T_{i}|\le (\ell-1)\binom{m}{k} < (\ell-1)m^{k}. \]
This contradicts our choice of $s$.
\end{proof}

\begin{proof}[Proof of \Cref{thm:perm}]
By \Cref{lemma:epsilon}, there is an $\eps>0$ such that any $P$-free $n\times n$ \zm matrix with at least $(1-\eps)n^2$ 0-entries contains an $\eps n\times \eps n$ all-0 submatrix. We can apply \Cref{lemma:dense_sparse} with this $\eps$ to get some $\delta>0$ such that any $P$-free $n\times n$ \zm matrix has a $\delta n\times \delta n$ submatrix $B$ with at least $(1-\eps)(\delta n)^2$  entries that are all 0 or all 1.

Let $A$ be an $n\times n$ \zm matrix that is both $P$-free and $P^{c}$-free, and let $B$ be the $\delta n\times \delta n$ submatrix with at least $(1-\eps)(\delta n)^2$ equal entries. If these entries are all 0, then $B$ contains an $\eps\delta n\times \eps\delta n$ all-0 submatrix because it is $P$-free. Otherwise, $B^{c}$ is a $P$-free matrix with at least $(1-\eps)(\delta n)^2$ 0-entries, so $B$ contains an $\eps\delta n\times \eps\delta n$ all-1 submatrix.
\end{proof}

\section{Unordered matrices--Proof of \Cref{thm:graph}}\label{sec:graph}

In this section, we prove \Cref{thm:graph}. Again, we show that if an unordered $P$-free matrix has a positive density of 0-entries, then it contains a linear-size all-0 submatrix. 

\begin{lemma} \label{lemma:unordered}
	Let $P$ be a simple $2\times k$ \zm matrix. Then every \emph{unordered} $P$-free $n\times n$ \zm matrix with at least $\eps n^2$ 0-entries contains an $\frac{\eps n}{6k}\times \frac{\eps n}{6k}$ all-0 submatrix.
\end{lemma}

\begin{proof}
	Let $R$ be the $2\times (2k+2)$ matrix whose first $k$ columns are $\begin{pmatrix}1\\0\end{pmatrix}$, the next $k$ columns are $\begin{pmatrix}0\\1\end{pmatrix}$, the $(2k+1)$'st column is $\begin{pmatrix}1\\1\end{pmatrix}$, and the last column is $\begin{pmatrix}0\\0\end{pmatrix}$. Then $R$ contains an ordering of the columns of $P$, so it is enough to prove our result for $R$ instead of $P$.
	
	Let $A'$ be the matrix obtained from $A$ by deleting the rows with fewer than $\eps n/2$ 0-entries. At most $\eps n^2/2$ 0's are deleted, so $A'$ contains at least $\eps n^2/2$ 0-entries. Hence, one can find a column in $A'$ with $t=\lceil\eps n/2\rceil$  0 entries. Let $B$ be the $t\times n$ submatrix of $A$ induced by the rows of these 0-entries.
	
	By permuting rows and columns if necessary, we may assume that these 0-entries form an all-0 last column in $B$, and that the rows of $B$ are in increasing order according to the number of 0-entries in them.
	
	For $i\in[t]$, let $H_{i}$ denote the set of indices $j\in [n]$ such that $B(i,j)=0$. Define the directed graph $G$ on vertex set $[t]$ by adding $(i,j)$ as an edge if $i<j$ and $|H_{i}\setminus H_{j}|\le k-1$. Then $G$ is an acyclic directed graph.
	
	Note that if $(i,j)$ is not an edge of $G$ for some $i<j$, then we must have $H_{i}\cup H_{j}=[n]$. Indeed, if $r\in [n]\setminus (H_{i}\cup H_{j})$, then $B[\{i,j\}\times \{r\}]=\begin{pmatrix}1\\1\end{pmatrix}$.
	We also have $|H_{i}\setminus H_{j}|\ge k$, which further implies $|H_{j}\setminus H_{i}|\ge k$ because $|H_{i}|\le |H_{j}|$. Therefore, if $X$ is a $k$-element subset of $H_{i}\setminus H_{j}$, and $Y$ is a $k$-element subset of $H_{j}\setminus H_{i}$, then $B[\{i,j\}\times (X\cup Y\cup\{r,n\})]$ is a reordering of $R$, contradicting our assumption.
	
	\medskip
	For a set $Z\subs [n]$, we denote its complement by $\overline{Z}=[n]\setminus Z$. Let $M$ be the set of minimal vertices in $G$, that is, the set of vertices $v$ such that no edge points towards $v$. Then the sets $\overline{H}_{v}$ are pairwise disjoint for $v\in M$. Every element $w\in [t]$ can be reached from a minimal vertex via a directed path. Let us assign each $w$ to the one such vertex in $M$ with the smallest label in $[t]$.
	
	Now we will show that there is a subset $N\subs M$ such that $|\bigcup_{v\in N} \overline{H}_{v}|\leq n-t$ and at least $t/3$ of the elements in $[t]$ are assigned to vertices in $N$. Note that by the construction of $B$, we have $|H_{i}|\ge t$ and hence $|\overline{H_{i}}|\leq n-t$ for every $i\in [t]$. If $M$ contains a vertex $v$ that is assigned to more than $t/3$ elements of $[t]$, then we are done, as we can take $N=\{v\}$. So we may assume that there is no such vertex. Starting with $N_{0}=\emptyset$, add vertices of $M$ one by one to $N_{0}$ until the number of elements assigned to the vertices in $N_{0}$ is at least $t/3$. At this point, the number of elements assigned to $N_{0}$ is between $t/3$ and $2t/3$. If $|\bigcup_{v\in N_{0}}\overline{H}_{v}|\leq n-t$, then set $N=N_{0}$, otherwise, set $N=M\setminus N_{0}$. As $t=\lceil \eps n/2\rceil\leq \lceil n/2\rceil$, we must have $|\bigcup_{v\in N}\overline{H}_{v}|\le n - |\bigcup_{v\in N_{0}}\overline{H}_{v}| \le t-1 \le n-t$. The number of elements assigned to an element of $N$ is at least $t/3$ in both cases.
	
	\medskip
	Now let $x_{1}<\dots<x_{s}$ be the elements of $[t]$ assigned to $N$, so $s\geq t/3$. Also, for $X=\bigcap_{v\in N} H_{v}$, we have $|X|\geq t$. We show by induction on $\ell$ that $|X\cap \bigcap_{i=1}^{\ell}H_{x_i}|\geq |X|-(\ell-1)(k-1)$. If $\ell=1$, then $x_{1}$ is a minimal element, so $X\subs H_{x_{1}}$, and we are done. Now suppose that $\ell>1$. If $x_{\ell}\in N$, then $X\subs H_{x_{\ell}}$, so
	\[ \left|X\cap \bigcap_{i=1}^{\ell}H_{x_{i}}\right|=\left|X\cap \bigcap_{i=1}^{\ell-1}H_{x_{i}}\right|\geq |X|-(\ell-2)(k-1)\geq |X|-(\ell-1)(k-1),\]
	and we are done. If $x_{\ell}\not\in N$, then $G$ must contain an edge $(x_{\ell'},x_{\ell})$ for some $1\leq \ell'<\ell$. Indeed, if $x_{\ell}$ is assigned to $v\in N$, then all other vertices on a $v$-$x_{\ell}$ directed path are assigned to $v$, as well. Now we can use $|H_{x_{\ell'}}\setminus H_{x_{\ell}}|\le k-1$, and hence $|(X\cap \bigcap_{i=1}^{\ell-1}H_{x_{i}})\setminus H_{x_{\ell}}|\leq k-1$, to get
	\[ \left|X\cap \bigcap_{i=1}^{\ell}H_{x_{i}}\right|\geq \left|X\cap \bigcap_{i=1}^{\ell-1}H_{x_{i}}\right|-(k-1)\geq |X|-(\ell-1)(k-1). \]
	
	Fix $\ell=\min\{\frac{t}{2(k-1)},\frac{t}{3}\}$ (for $k=1$, take $\ell=\frac{t}{3}$). Then $|\bigcap_{i=1}^{\ell}H_{x_{i}}|\geq |X|-(\ell-1)(k-1)\geq t/2$, so the submatrix of $B$ induced by the rows $\{x_{1},\dots,x_{\ell}\}$ and columns $\bigcap_{i=1}^{\ell}H_{x_{i}}$ is an all-0 matrix with at least $\min\{\frac{t}{2(k-1)},\frac{t}{3}\}\ge \frac{\eps n}{6k}$ rows, and at least $\frac{t}{2}\ge \frac{\eps n}{4}$ columns. This finishes the proof.
\end{proof}

\begin{proof}[Proof of \Cref{thm:graph}]
	If $A$ has at least $\frac{n^{2}}{2}$ 0-entries, we can find a $\frac{n}{12k}\times \frac{n}{12k}$ all-0 submatrix in $A$ by the previous lemma. Otherwise, we can apply \Cref{lemma:unordered} to $A^{c}$ to show that $A$ contains a $\frac{n}{12k}\times \frac{n}{12k}$ all-1 submatrix.
\end{proof}

\Cref{lemma:unordered} shows that there is a genuine difference between the ordered and unordered case of our problem. Indeed, this result shows that $\eps n^2$ 0-entries in an \emph{unordered} $P$-free matrix guarantee an $\Omega(\eps n)\times \Omega(\eps n)$ all-0 submatrix. However, as we discussed at the end of \Cref{sec:chessboard}, this is not true for every $2\times k$ matrix $P$ in the ordered setting: there are $P$-free matrices with $\eps n^2$ 0-entries that do not have any all-0 submatrix of size $\Omega(\frac{\eps}{\alpha(1/\eps)}n)$.

By a result of F\"uredi \cite{F}, there is an $n\times n$ matrix $A$ with $\Theta(n\log n)$ 0-entries that does not contain $\begin{pmatrix}0&*&0\\ 0&0&*\end{pmatrix}$ (where $*$ can be either 1 or 0). With the same methods as before, we can use this to construct $n\times n$ matrices $A$ with $\eps n^2$ 0-entries that do not contain $\begin{pmatrix}0&1&0\\0&0&1\end{pmatrix}$ and have no all-0 submatrices of size $\Omega(\frac{\eps}{\log 1/\eps}n)$.

\section{Applications} \label{sec:applications}

Several matrix classes can be described by a finite set of forbidden submatrices (see, e.g., \cite{KRW95}), and our results show that in many cases they contain large homogeneous submatrices. We give three specific applications.

\subsection{Chordal bipartite graphs and totally balanced matrices} \label{sec:chordal}

A \zm matrix is \emph{totally balanced} if it does not contain any submatrix, whose columns are different and which has exactly two 1-entries in each of its rows and columns. In other words, none of its submatrices is the incidence matrix of a cycle of length at least 3.

Totally balanced matrices (first studied by Lov\'asz \cite{L68} in connection with a hypergraph coloring problem) are well-examined objects in combinatorial optimization. Their importance comes from the fact that integer programs with totally balanced coefficient matrices can be easily solved. Indeed, the optimization problem can be solved greedily if the coefficient matrix does not contain $\Gamma=\begin{pmatrix} 1&1\\1&0\end{pmatrix}$ as a submatrix, and as was shown in \cite{AF84,HKS85,L87}, a matrix is totally balanced if and only if its rows and columns can be rearranged to get a $\Gamma$-free matrix. (For more on optimization properties of balanced matrices, see the book of Berge \cite{BBOOK}.) As rearranging rows and columns does not affect homogeneous submatrices, \Cref{thm:2} shows that totally balanced matrices have large homogeneous submatrices.

\begin{corollary}
Every totally balanced $n\times n$ matrix contains an $cn \times cn$ homogeneous submatrix with some $c\ge 1/20$.
\end{corollary}

A \emph{chordal bipartite graph} is a bipartite graph with no induced cycle of length greater than 4. This class of graphs was introduced by Golumbic and Goss \cite{GG78} as a bipartite analog to chordal graphs, with similar perfect elimination properties. Clearly, a bipartite graph is chordal if and only if its adjacency matrix is totally balanced. This immediately implies the following.

\begin{corollary}
Every chordal bipartite graph $G=(A\cup B,E)$ with parts of size $n$ contains sets $A'\subs A$ and $B'\subs B$ of size $cn$, for some constant $c>0$, such that $G[A',B']$ is either empty or complete.
\end{corollary}

\subsection{The Erd\H{o}s-Hajnal conjecture and intersection graphs} \label{sec:geometry}
A family $\mathcal{G}$ of graphs is said to have the \emph{Erd\H{o}s-Hajnal property}, if there is a constant $c$ such that each member $G\in\mathcal{G}$ contains a clique or an independent set on at least $|V(G)|^{c}$ vertices. The family $\mathcal{G}$ has the \emph{strong Erd\H{o}s-Hajnal property}, if there is a constant $c'$ such that every $G\in\mathcal{G}$ satisfies that either $G$ or its complement contains a complete bipartite graph with parts of size $c'|V(G)|$. By a result of Alon, Pach, Pinchasi, Radoi\v{c}i\'c and Sharir \cite{APPRS}, the strong Erd\H{o}s-Hajnal property implies the Erd\H{o}s-Hajnal property in hereditary families. The famous Erd\H{o}s-Hajnal conjecture \cite{EH1,EH2} asserts the following.

\begin{conjecture}[Erd\H{o}s, Hajnal]
	For every graph $H$, the family of graphs not containing an induced copy of $H$ has the Erd\H{o}s-Hajnal property.
\end{conjecture}

This conjecture has attracted significant attention in the past decades, but is still wide open. For history and relevant results, we refer the reader to the survey of Chudnovsky \cite{C}.

\medskip
The \emph{intersection graph} of a family of sets $\mathcal{F}$ is the graph with vertex set $\mathcal{F}$, where two vertices are joined by an edge if their intersection is nonempty.  A \emph{curve} in the plane is the image of an injective continuous function $f:[0,1]\rightarrow \mathbb{R}^{2}$.  In this paper, we assume that curves in our collections only meet at proper crossings, that is, if two curves $\alpha$ and $\beta$ share a point in common, then $\alpha$ passes to the other side of $\beta$ at this point. A \emph{string graph} is a graph that is isomorphic to the intersection graph of a family of curves.

In a very recent paper, Tomon \cite{T} showed that the family of string graphs has the Erd\H{o}s-Hajnal property. However, this family does not satisfy the strong Erd\H{o}s-Hajnal property \cite{PT06}, although Fox and Pach \cite{FP2} proved that one can always find a complete bipartite graph of almost linear size in every string graph or its complement.

\begin{theorem}[Fox, Pach]
Let $G$ be a string graph on $n$ vertices. Then either $G$ contains $K_{m,m}$ with $m=\Omega(\frac{n}{\log n})$, or the complement of $G$ contains $K_{m',m'}$ with $m'=\Omega(n)$.
\end{theorem}

A collection of curves is \emph{$k$-intersecting}, if any two curves in the collection intersect in at most $k$ points. Fox, Pach and T\'oth \cite{FPT} showed that the family of intersection graphs of $k$-intersecting curves does have the strong Erd\H{o}s-Hajnal property.

\begin{theorem}[Fox, Pach, T\'oth] \label{thm:kintersect}
For every positive integer $k$, there is a constant $c_{k}>0$ such that the following holds. Let $G$ be the intersection graph of a $k$-intersecting family of $n$ curves. Then either $G$ or its complement contains a complete bipartite graph of size $c_{k}n$.
\end{theorem}

Here, we are interested in a bipartite version of this problem. That is, given two families of $n$ curves, $\mathcal{A}$ and $\mathcal{B}$, we would like to find large subfamilies, $\mathcal{A}_{0}\subs \mathcal{A}$ and $\mathcal{B}_{0}\subs \mathcal{B}$, such that $|\mathcal{A}_{0}|=|\mathcal{B}_{0}|$, and either every curve in $\mathcal{A}_{0}$ intersects every curve in $\mathcal{B}_{0}$, or every curve in $\mathcal{A}_{0}$ is disjoint from every curve in $\mathcal{B}_{0}$.

In general, we cannot hope for any bound on $|\mathcal{A}_{0}|=|\mathcal{B}_{0}|$ beating the Ramsey bound $\Theta(\log n)$. Indeed, the complement of every comparability graph is a string graph \cite{L83,PT06}, therefore the complement of any bipartite graph is a string graph.  Nevertheless, the question remains meaningful if we restrict ourselves to $k$-intersecting collections of curves.

In fact, we believe that the condition that $\mathcal{A}\cup\mathcal{B}$ is $k$-intersecting can be weakened to only requiring that $\mathcal{A}$ and $\mathcal{B}$ themselves are $k$-intersecting.

\begin{conjecture}\label{conj:curves2}
For every $k$ there is a constant $c_k>0$ such that the following holds. Let $\mathcal{A}$ and $\mathcal{B}$ be two families of $n$ curves each such that $\mathcal{A}$ and $\mathcal{B}$ are $k$-intersecting. Then there are subfamilies $\mathcal{A}_0\subs \mathcal{A}$ and $\mathcal{B}_0\subs \mathcal{B}$ such that $|\mathcal{A}_0|=|\mathcal{B}_0|\ge c_k n$, and either every $\alpha\in \mathcal{A}_0$ intersects every $\beta\in \mathcal{B}_0$, or every $\alpha\in \mathcal{A}_0$ is disjoint from every $\beta\in \mathcal{B}_0$.
\end{conjecture}

In some sense, this is the weakest condition one can impose on $\mathcal{A}$ and $\mathcal{B}$ to force any meaningful properties. Indeed, the complement of any bipartite graph can be realized as the intersection graph of a collection of curves $\mathcal{A}\cup\mathcal{B}$, where $\mathcal{A}$ is $1$-intersecting, and any two curves $A\in\mathcal{A}$ and $B\in\mathcal{B}$ intersect in at most 2 points (but $\mathcal{B}$ is not $k$-intersecting for any bounded $k$), see \cite{PaT}.

A natural special case of the conjecture  is when the curves are 0-1 curves. Here, a \emph{0-1 curve} is the drawing of a continuous function $f:[0,1]\rightarrow \mathbb{R}$ in $\mathbb{R}^{2}$. As a first step towards \Cref{conj:curves2}, we prove the following statement.

\begin{theorem} \label{thm:intsect}
Let $\mathcal{A}$ and $\mathcal{B}$ be two families of $n$ 0-1 curves each. If $\mathcal{A}$ is $k$-intersecting, and $\mathcal{B}$ is 1-intersecting, then there are subfamilies $\mathcal{A}_{0}\subs \mathcal{A}$ and $\mathcal{B}_{0}\subs \mathcal{B}$ such that $|\mathcal{A}_{0}|=|\mathcal{B}_{0}|\geq \Omega(n/k)$, and either every $\alpha\in \mathcal{A}_0$ intersects every $\beta\in \mathcal{B}_0$, or every $\alpha\in \mathcal{A}_0$ is disjoint from every $\beta\in \mathcal{B}_0$.
\end{theorem}
\begin{proof}
By slightly perturbing our curves, we can assume that no 3 curves in $\mathcal{A}\cup\mathcal{B}$ go through the same point, and no two of them intersect the lines $x=0$ and $x=1$ in the same point. For two curves $\gamma,\gamma'\in\mathcal{A}\cup\mathcal{B}$, let $\gamma\prec\gamma'$ if $\gamma$ intersects the vertical line $x=0$ below $\gamma'$.

First, we claim that there are subfamilies $\mathcal{A}'\subs\mathcal{A}$ and $\mathcal{B}'\subs\mathcal{B}$ such that $|\mathcal{A}'|=|\mathcal{B}'|=\lceil n/2\rceil$, and either $\alpha\prec\beta$ for every $(\alpha,\beta)\in\mathcal{A}'\times\mathcal{B}'$, or $\beta\prec\alpha$ for every $(\alpha,\beta)\in\mathcal{A}'\times\mathcal{B}'$. Indeed, in the total ordering defined by $\prec$, pick the smallest element  $\gamma\in\mathcal{A}\cup\mathcal{B}$ such that either $\lceil n/2\rceil$ elements of $\mathcal{A}$ are $\preceq \gamma$, or $\lceil n/2\rceil$ elements of $\mathcal{B}$ are $\preceq \gamma$. In the first case, set $\mathcal{A}'=\{\alpha\in\mathcal{A}:\alpha\preceq \gamma\}$ and let $\mathcal{B}'$ be an $\lceil n/2\rceil$ element subset of $\{\beta\in\mathcal{B}:\gamma\prec\beta\}$. In the second case, let  $\mathcal{A}'$ be an $\lceil n/2\rceil$ element subset of $\{\alpha\in\mathcal{A}:\gamma\prec\alpha\}$ and $\mathcal{B}'=\{\beta\in\mathcal{B}:\beta\preceq\gamma\}$.

Without loss of generality, suppose that $\alpha\prec\beta$ for every $(\alpha,\beta)\in\mathcal{A}'\times\mathcal{B}'$. Define the $\lceil n/2\rceil\times \lceil n/2\rceil$ matrix $A$ by setting $A(i,j)=1$ if the $i$'th smallest element of $\mathcal{A}'$ intersects the $j$'th smallest element of $\mathcal{B}'$ with respect to the ordering $\prec$, and $A(i,j)=0$ otherwise.

\begin{claim}\label{claim:kintersect}
Let $P_{\ell}$ be the $2\times \ell$ matrix defined by $P_{\ell}(i,j)=1$, if $i+j$ is even, and $P_{\ell}(i,j)=0$ if $i+j$ is odd. Then $A$ is $P_{k+2}$-free.
\end{claim}
\begin{proof}
Let us start with introducing some notation. Each 0-1 curve cuts the strip $[0,1]\times\mathbb{R}$ into two parts, an upper and lower part. We say that a point set is \emph{above} the curve if it is a subset of the upper part, and it is \emph{below}, if it is a subset of the lower part. Also, if $\gamma$ is a 0-1 curve and $q\in\gamma$, let $\gamma(q)$ denote the subcurve of $\gamma$ starting on the vertical line $x=0$, and ending at $q$. For $q,q'\in\gamma$, we define $\gamma(q,q')=\gamma(q')\setminus\gamma(q)$.

Suppose that $\alpha\prec\alpha'$ in $\mathcal{A}$, and $\beta_{1}\prec\dots\prec\beta_{k+2}$ in $\mathcal{B}$ induce $P_{k+2}$. Let $p_{1},\dots,p_{t}$ be the intersection points of the curves $\alpha$ and $\alpha'$, ordered by their $x$-coordinates. As $\mathcal{A}$ is $k$-intersecting, we have $t\le k$. These $t$ intersection points cut both $\alpha$ and $\alpha'$ into $k+1$ subcurves, let us denote them by $\alpha_{0},\dots,\alpha_{t}$ and $\alpha'_{0},\dots,\alpha'_{t}$ from left to right. Note that $\alpha<\alpha'$ implies that if $i$ is even, then $\alpha_{i}$ is below $\alpha'$, and $\alpha_{i}'$ is above $\alpha$, while if $i$ is odd, then $\alpha_{i}$ is above $\alpha'$ and $\alpha_{i}'$ is below $\alpha$. For $i=0,\dots,t$, let $L_{i}$ denote the region in $[0,1]\times\mathbb{R}$ bounded by $\alpha_{i}$ and $\alpha_{i}'$,  and call these regions $L_{i}$ \emph{lenses}. If $i$ is even, say that $\alpha_{i}'$ is the \emph{top boundary} of $L_{i}$ and $\alpha_{i}$ is the \emph{bottom boundary}, and if $i$ is odd, then $\alpha_{i}$ is the top boundary of $L_{i}$, and $\alpha_{i}'$ is the bottom boundary. Note that if $\beta_{j}$ intersects the lens $L_{i}$, then $\beta_{j}$ intersects only the top boundary of $L_{i}$, as $\alpha,\alpha'\prec\beta_{i}$ and each of the curves $\beta_{i}$ intersect exactly one of $\alpha$ and $\alpha'$.  Therefore, if $L_{i}$ and $\beta_{j}$ intersect, $i$ and $j$ must have the same parity.

\begin{figure}[t!]
\includegraphics[trim=0 5cm 0 1.5cm, clip, scale=.7]{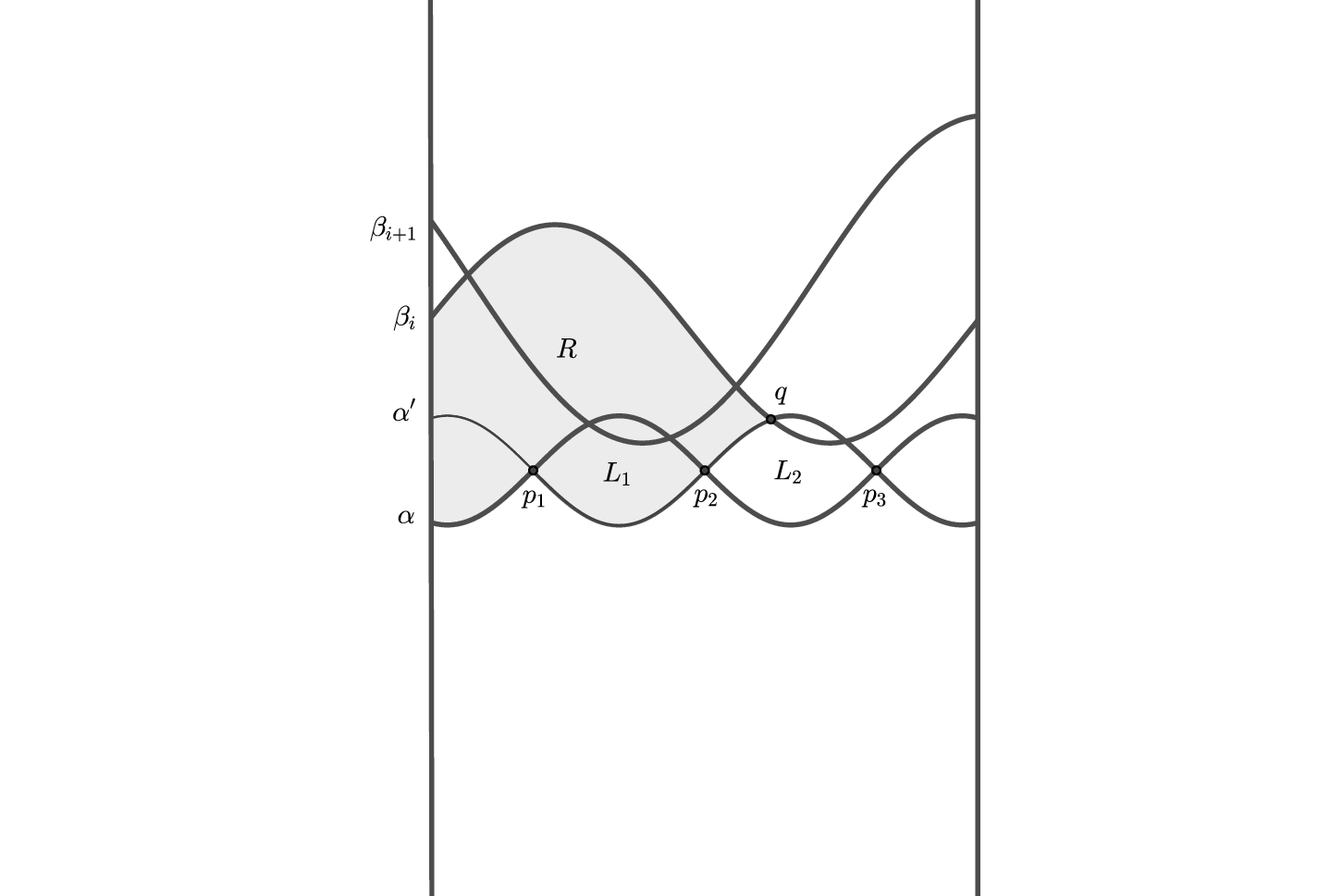}
\caption{An illustration of the proof of \Cref{claim:kintersect} }
\label{image1}
\end{figure}

For $i\in [k+2]$, let $\ell(i)$ denote the smallest index for which $\beta_{i}$ intersects the lens $L_{\ell(i)}$. We show that $\ell(1)<\ell(2)<\dots<\ell(k+2)$, which contradicts $0\leq \ell(i)\leq k$.  Suppose that $\ell(i+1)\leq \ell(i)$ for some $i\in [k+1]$. As $i$ and $i+1$ have different parities, we have $\ell(i+1)<\ell(i)$ and $\beta_{i+1}$ cannot intersect $L_{\ell(i)}$. Let $\gamma$ denote the union of the top boundaries of all the lenses, and let $\gamma'$ denote the union of the bottom boundaries of the lenses, then $\gamma$ and $\gamma'$  are 0-1 curves. Let $q$ be the first intersection point of $\gamma$ and $\beta_{i}$, and let
\[ \delta=\gamma'(p_{\ell(i)})\cup \gamma(p_{\ell(i)},q). \]
In other words, we obtain the curve $\delta$ by following the bottom boundaries of the lenses until we reach the lens $L_{\ell(i)}$, where we follow the top boundary until we reach $\beta_{i}$. Let $R$ be the region bounded by $\beta_{i}(q)$ and $\delta$, see \Cref{image1}. The curve $\beta_{i+1}$ starts outside $R$, but $R$ contains the lens $L_{\ell(i+1)}$, so $\beta_{i+1}$ must enter $R$. However, $\beta_{i+1}$ does not intersect intersect $\gamma'$, nor does it touch $L_{\ell(i)}$. Thus, $\beta_{i+1}$ cannot intersect $\delta$. Hence, $\beta_{i+1}$ must enter $R$ through $\beta_{i}(q)$. Since $\beta_{i+1}$ also leaves $R$, it must also exit through $\beta_{i}(q)$. Therefore, $\beta_{i}$ and $\beta_{i+1}$ intersect twice, contradiction.
\end{proof}

The $2\times k$ matrix $P_{k+2}$ does not contain a homogeneous column, so we can apply \Cref{thm:1} to conclude that $A$ contains a homogeneous submatrix of size at least $\Omega(n/k)$. This corresponds to two collections $\mathcal{A}_{0}\subs \mathcal{A}$ and $\mathcal{B}_{0}\subs \mathcal{B}$ with the desired properties.

 \end{proof}

\subsection{Pseudohalfplanes}

A \emph{bi-infinite $x$-monotone curve} is the graph of a continuous function $f:\mathbb{R}\rightarrow \mathbb{R}$. A collection $\mathcal{L}$ of bi-infinite $x$-monotone curves is a \emph{pseudoline-arrangement} if any two elements of $\mathcal{L}$ intersect in exactly one point. If $\mathcal{L}$ is a pseudoline-arrangement, then $\mathcal{H}$ is a \emph{pseudohalfplane-arrangement} if every element $H\in\mathcal{H}$ is either the set of points below an element of $\mathcal{L}$, or the set of points above an element of $\mathcal{L}$. 

Let $P$ be a set of points in the plane  and let $\mathcal{H}$ be a pseudohalfplane-arrangement. Consider the matrix $M$ whose rows are labeled with elements of $P$, columns are labeled with the elements of $\mathcal{H}$, and 
\[ 
M(p,H)=\begin{cases}1 &\mbox{ if } p\in H\\
0 &\mbox{ if }p\not\in H
\end{cases}. 
\]
It is proved in \cite[Theorem 2.19, Proposition A.1]{KP19} (see also \cite{BHP}) that $M$ can be partitioned into two submatrices $M_{1}$ and $M_{2}$ such that the following holds: the rows and columns of $M_{1}$ and $M_{2}$ can be ordered such that  $M_{1}$ and $M_{2}$ does not contain  $\begin{pmatrix}1 & 0  \\ 0 & 1 \end{pmatrix}$ as a submatrix. But then \Cref{thm:2by2} immediately implies that some linear set of pseudohalfplanes contains or avoids a positive proportion of the points.

\begin{corollary}\label{thm:pseudo}
	Let $P$ be a set of $n$ points in the plane and let $\mathcal{H}$ be a pseudohalfplane-arrangement with $n$ elements. Then there are subsets $P_{0}\subset P$ and $\mathcal{H}_{0}\subset \mathcal{H}$ of size $|P_{0}|=|\mathcal{H}_{0}|\geq cn$ for a suitable constant $c>0$, such that either for every $p\in P_{0}$ and $H\in\mathcal{H}_{0}$ we have $p\in H$, or for every $p\in P_{0}$ and $H\in\mathcal{H}_{0}$ we have $p\not\in H$.
\end{corollary}

\section{Concluding remarks}\label{sec:remarks}

\begin{figure}[t!]
\begin{center}
\begin{tabular}{cccc}

 $P$ & \Cref{conj:acyclic} & Best bound &  Reference \\

 \noalign{\smallskip}
 \hline
 \noalign{\smallskip}
 
 $\begin{pmatrix}0 & 0  \\ 0 & 0 \end{pmatrix}$ & \Checkmark & $\ge cn\times cn$ & \Cref{lem:homcolumn} \\
 \noalign{\smallskip}
 
 $\begin{pmatrix}1 & 0  \\ 0 & 0 \end{pmatrix}$ & \Checkmark & $\ge cn\times cn$ & \Cref{lemma:1cornereasy} \\
 \noalign{\smallskip}

 $\begin{pmatrix}1 & 1  \\ 0 & 0 \end{pmatrix}$ & \Checkmark & $\ge cn\times cn$ & \Cref{lemma:longmtx} \\
 \noalign{\smallskip}

 $\begin{pmatrix}1 & 0  \\ 0 & 1 \end{pmatrix}$ & \Checkmark & $\ge cn\times cn$ & \Cref{lemma:longmtx} \\
 \noalign{\smallskip}

 $\begin{pmatrix}1 & 1  \\ 1 & 0 \end{pmatrix}$ & \Checkmark & $\ge cn\times cn$ & \Cref{lemma:1corner} \\
 \noalign{\smallskip}

 $\begin{pmatrix}1 & 1  \\ 1 & 1 \end{pmatrix}$ & $\times$   & $\le n^{1-\eps}\times n^{1-\eps}$   & \Cref{prop:cyclicmx} \\
 \noalign{\smallskip}

 $\begin{pmatrix}1 & 1 & 0 \\ 0 & 0 & 0 \end{pmatrix}$ & \Checkmark & $\ge cn\times cn$ & \Cref{lemma:1cornereasy} \\
 \noalign{\smallskip}

 $\begin{pmatrix}1 & 0 & 1 \\ 0 & 1 & 0 \end{pmatrix}$ & \Checkmark & $\ge cn\times cn$ & \Cref{lemma:longmtx} \\
 \noalign{\smallskip}

 $\begin{pmatrix}1 & 1 & 1 \\ 1 & 0 & 0 \end{pmatrix}$ & ? & $\ge \frac{cn}{\log n}\times cn$ & \Cref{lemma:ordered} \\
 \noalign{\smallskip}

 $\begin{pmatrix}1 & 1 & 1 \\ 0 & 1 & 0 \end{pmatrix}$ & ? & $\ge n^{1-o(1)}\times cn$ & \Cref{lemma:ordered} \\
 \noalign{\smallskip}

 $\begin{pmatrix}1 & 0 & 0 \\ 0 & 1 & 0 \\ 0 & 0 & 1 \end{pmatrix}$ & ? & 
 
 \begin{tabular}{c}
 $\ge cn\times cn$ \\ {\footnotesize if $\ge (1-\eps) n^2$ 0-entries}
 \end{tabular}
 
 & \Cref{lemma:epsilon} \\

\end{tabular}

\end{center}

\caption{Largest size of an all-0 submatrix in an $n\times n$ $P$-free matrix with at least $\eps n^2$ 0-entries.}
\label{fig:table}
\end{figure}

Our work establishes various bounds on the size of the largest homogeneous submatrix that can be found in a matrix, when a fixed submatrix $P$ is forbidden. A summary of our results for fixed small $P$ can be found in \Cref{fig:table}. A number of questions remain unsolved, and it would be very interesting to obtain good bounds for simple or acyclic matrices. Perhaps the first open question is to decide if $Q_2 = \begin{pmatrix}1 & 1 & 0  \\ 0 & 0 & 0\end{pmatrix}$ satisfies \Cref{conj:acyclicpair}, i.e., if forbidding the submatrix $Q_2$ in an $n\times n$ \zm matrix guarantees the existence of a $cn\times cn$ homogeneous submatrix.

\medskip

These questions are also closely related to recent results on the Erd\H{o}s-Hajnal theory of trees: Extending previous work in \cite{BLT,LPSS},
Chudnovsky, Scott, Seymour and Spirkl \cite{CSSS} proved the following variant of the Erd\H{o}s-Hajnal conjecture.

\begin{theorem}[Chudnovsky et al.]\label{thm:chudnovsky}
	Let $T$ be a tree. Then the family of all graphs not containing an induced copy of $T$ and $T^{c}$ has the strong Erd\H{o}s-Hajnal property.
\end{theorem}

Our problems can be thought of as an ordered bipartite version of the strong Erd\H{o}s-Hajnal problem. For example, it is not hard to see that \Cref{thm:chudnovsky} would follow from \Cref{conj:acyclicpair}.

Indeed, we can think of our $n\times n$ \zm matrix $A$ as the biadjacency matrix of a bipartite graph $G(A\cup B,E)$ with parts of size $n$. Submatrices then correspond to induced subgraphs, and a homogeneous submatrix means a subgraph $G'=(A'\cup B',E')$ that is complete or empty between $A'\subs A$ and $B'\subs B$. An important difference, though, is that a forbidden submatrix only forbids \emph{one ordering} of the corresponding bipartite graph (where the vertices in the two parts are ordered according to the rows and columns of the matrix). This is a much weaker condition and adds considerable difficulty to our problem.

Approximate versions of our \Cref{conj:acyclicpair,conj:acyclic}, finding $n^{1-o(1)} \times n^{1-o(1)}$ homogeneous submatrices, were very recently proved by Scott, Seymour and Spirkl \cite{SSS}.

\medskip

Extremal questions about \zm matrices have been extensively studied over the past decades, and it is worth mentioning a few that are loosely related to our problem.

A \zm matrix $A$ contains a \emph{pattern} $P$, where $P$ is another \zm matrix, if $P$ can be obtained from a submatrix of $A$ by changing some 1-entries to 0-entries. When $A$ is a biadjacency matrix, this corresponds to the subgraph relation (as opposed to submatrices corresponding to \emph{induced} subgraphs). The Tur\'an number $\ex(n,P)$ is defined as the maximum number of 1-entries in an $n\times n$ \zm matrix that does not contain the pattern $P$. A central problem in this area is a conjecture of Pach and Tardos \cite{PTa06} that $\ex(n,P) = O(n\polylog n)$ whenever $P$ is acyclic. Although this is known for many such matrices \cite{FH,MT04,PTa06,KTTW19}, the general conjecture remains open.

Another related question asks for $\forb(n,P)$, the maximum number of distinct columns in an unordered $P$-free \zm matrix $A$ with $n$ rows. When we think of $A$ as the incidence matrix of a hypergraph, finding $\forb(n,P)$ is connected to certain hypergraph coloring problems (see, e.g., \cite{L68}), as well as other structural results. For example, in the special case when $P$ is the $k\times 2^k$ \zm matrix with all different columns, the Sauer-Shelah lemma gives $\forb(n,P) = \binom{n}{k-1} + \binom{n}{k-2} +\dots + \binom{n}{0}$. An open conjecture of Anstee and Sali \cite{AS05} asserts that $\forb(n,P) = \Theta(n^{f(P)})$ for an implicitly defined integer function $f$. For further partial results on this topic, we refer the reader to the survey \cite{A13}.

\subsubsection*{Acknowledgments}

We thank Bal\'azs Keszegh and D\"om\"ot\"or P\'alv\"olgyi for drawing our attention to their recent paper \cite{KP19} and for pointing out that our \Cref{thm:2by2} implies \Cref{thm:pseudo}. We are also grateful to Maria Axenovich for sharing with us her manuscript \cite{ATW}.

\end{document}